\documentclass[11pt]{amsart}
\usepackage{amsthm,amsmath,amssymb}
\usepackage[utf8]{inputenc}
\usepackage[english]{babel}
\usepackage[T1]{fontenc}
\usepackage{indentfirst}

\usepackage[shortlabels]{enumitem}
\usepackage{graphicx}
\usepackage[dvipsnames]{xcolor}
\usepackage[pagebackref,breaklinks,colorlinks=true,linkcolor=MidnightBlue,citecolor=MidnightBlue]{hyperref}

\theoremstyle{definition}
\newtheorem{definition}{Definition}

\newtheorem{assumption}{Assumption}

\theoremstyle{plain}
\newtheorem{theorem}{Theorem}

\newtheorem{prop}{Proposition}

\newtheorem{fact}{Fact}
\newtheorem{correspondence}{Correspondence Principle}

\theoremstyle{remark}

\newtheorem{example}{Example}

\newcommand{\abs}[1]{\lvert #1 \rvert}

\newcommand{\set}[1]{\left\{\, #1 \,\right\}}
\newcommand{\bracket}[1]{\langle #1 \rangle}

\newcommand{\Real}{\mathbf{R}}

\newcommand{\Rd}{\mathbf{R}^d}
\newcommand{\Rn}{\mathbf{R}^n}
\let\div\relax
\DeclareMathOperator{\div}{div}

\DeclareMathOperator{\laplacian}{\Delta}

\DeclareMathOperator*{\minimize}{minimize}

\DeclareSymbolFont{bbold}{U}{bbold}{m}{n}
\DeclareSymbolFontAlphabet{\mathbbold}{bbold}

\newcommand{\calV}{\mathcal{V}}
\newcommand{\calS}{\mathcal{S}}

\newcommand{\calP}{\mathcal{P}}

\newcommand{\eps}{\varepsilon}

\title{A geometric perspective on regularized optimal transport}
\date{\today}
\author{Flavien Léger}
\address{Courant Institute of Mathematical Sciences, New York University, 251 Mercer Street, NY 10012}
\email{leger@cims.nyu.edu}
\thanks{Support is gratefully acknowledged from NSF grants DMS-1211806 and DMS-1311833}

\begin{document}

\begin{abstract}
We present new geometric intuition on dynamical versions of regularized optimal transport. We introduce two families of variational problems on Riemannian manifolds which contain analogues of the Schrödinger bridge problem and the Yasue problem. We also propose an analogue of the Hopf--Cole transformation in the geometric setting. 
\end{abstract}

\maketitle



\section{Introduction}
The purpose of this work is to present new geometric intuition on variants of optimal transport: the  Schrödinger bridge and Yasue problems. Both problems can be seen as fluid-based formulations of regularized optimal transport. We will develop a new viewpoint, discussing them from the perspective of geometry in Wasserstein space. Recall that such geometry-based intuition has proven useful, for instance by recasting some classes of PDEs as gradient flows~\cite{jko98,ottocalculus,ags08}. While our arguments are purely formal at the PDE level (i.e. for the Schrödinger bridge and Yasue problems), they are rigorous in a finite-dimensional setting.

We introduce two families of variational problems on manifolds which contain a direct analogue of the Schrödinger bridge and the Yasue problems. The first family consists of optimal control problems, while the second one comprises Lagrangian mechanics problems. The geometric problems are defined on any Riemannian manifold and involve a potential function; when this manifold is the Wasserstein space and the potential is the entropy then the geometric problems become the optimal transport problems. Yet, in addition to providing a fresh perspective on the Schrödinger bridge and the Yasue problems, the geometric problems have more variety, for instance in the choice of potential functions different from the entropy. Finally, an striking feature of the two optimal transport variants is that their Euler--Lagrange equations can be reduced to a linear PDE via a Hopf--Cole transformation; as an initial effort to harvest information form our new viewpoint we present an generalization of the Hopf--Cole transformation to the geometric setting. 

The Schrödinger bridge problem dates back to Schrödinger (\cite{schrodinger31,schrodinger32,fol88}, see also the survey~\cite{leonardsurvey}), thus preceding modern formulations of optimal transport~\cite{villanibook1,villanibook2,bb00}. Connections between the two topics were made in~\cite{mik04, mt06, mt08, leo12, cgp16}. Also, the Schrödinger bridge problem has been studied as a stochastic optimal control problem~\cite{mik90, dai91}. This is an alternative approach to the one presented in this paper: roughly speaking the extra term which is not present in classical optimal transport has been classically interpreted as a noise term, while we view it as a Wasserstein gradient. Additionally, the Schrödinger bridge problem is equivalent to an entropic regularization of the two-marginal Monge--Kantorovich problem~\cite{rt98, gs10} which has gained popularity in recent years because of fast numerical methods~\cite{cut13}. On the other hand, the Yasue problem is a time-symmetric equivalent of the Schrödinger bridge problem; we call it Yasue following Carlen~\cite{carlenfarischapter}. It has been considered in the literature under different names~\cite{yasue81, carlenfarischapter, cgp16}.

The presentation is mainly divided into two parts: in Section~\ref{sec:models} we introduce our new geometric problems while in Section~\ref{sec:hopfcole} we give a geometric version of the Hopf--Cole transformation. 



\section{Background}
\subsection{Forward and backward Schrödinger bridge problems} \label{sec:schroedinger}
The (forward) Schrödinger bridge problem is the following variational problem
\begin{equation}
\label{eq:SB}
\tag{SB}
\begin{aligned}
& \minimize_{\rho,\,b} & & \int_0^1 \!\! \int\frac{1}{2}\abs{b_t(x)}^2\,\rho_t(x)\,dx\,dt \\
& \text{subject to:} & & \partial_t\rho_t + \div(\rho_t b_t) = \gamma\laplacian\rho_t \quad\text{for } 0<t<1 \,, \\
& & & \rho_0 = \mu\,, \; \rho_1 = \nu
\end{aligned}
\end{equation}
Here $\rho$ is a time-dependent density in the space of probability measures $\calP(\Rd)$ and we denote time-dependence with a subscript: $\rho_t=\rho(t,\cdot)\,$. The vector field $b=b_t(x)$ can be seen as a control used to drive $\rho_t$ from an initial density $\mu\in\calP(\Rd)$ (time $t=0$) to a final density $\nu\in\calP(\Rd)$ (time $t=1$). The parameter $\gamma > 0$ is a fixed diffusion coefficient.
 
Reversing the arrow of time in~\eqref{eq:SB} leads to the backward-in-time Schrödinger bridge problem~\eqref{eq:SBstar}; it amounts to change the sign of the diffusion coefficient $\gamma$ :
\begin{equation}
\label{eq:SBstar}
\tag{SB*}
\begin{aligned}
& \minimize_{\rho,\, b^*} & & \int_0^1 \!\! \int\frac{1}{2}\abs{b^*_t(x)}^2\,\rho_t(x)\,dx\,dt \\
& \text{subject to:} & & \partial_t\rho_t + \div(\rho_t b_t^*) =- \gamma\laplacian\rho_t \quad\text{for } 0<t<1 \,, \\
& & & \rho_0 = \mu\,, \; \rho_1 = \nu
\end{aligned}
\end{equation}

There is a priori no reason why the two problems~\eqref{eq:SB} and~\eqref{eq:SBstar} should be equivalent and it is indeed not intuitively clear, however is it known to be true, as explained in Section~\ref{sec:equivalence}. This phenomenon is closely connected to the time-reversibility property of the related stochastic processes~\cite{bernstein32}.

\subsection{The time-symmetric Yasue problem} \label{sec:yasue}
The Yasue problem is
\begin{equation}
\label{eq:Y}
\tag{Y}
\begin{aligned}
& \minimize_{\rho,\,v}   & & \int_0^1 \!\! \int\frac{1}{2}\abs{v_t(x)}^2\,\rho_t(x) + \frac{\gamma^2}{2}\frac{\abs{\nabla \rho_t(x)}^2}{\rho_t(x)}\,dx \,dt \\
& \text{subject to:} & & \partial_t\rho_t + \div(\rho_t v_t) = 0 \quad\text{for } 0<t<1 \,, \\
& &  & \rho_0 = \mu\,, \; \rho_1 = \nu
\end{aligned}
\end{equation}
Compared to the Schrödinger bridge problem, it is not so much an optimal control problem, but rather a (infinite-dimensional) Lagrangian mechanics problem: one looks for critical points of an action functional $\mathcal{A}=\int_0^1 \! \int\frac{1}{2}\abs{v}^2\,\rho + \frac{\gamma^2}{2}\abs{\nabla \rho}^2/\rho\,dx\,dt\,$ which includes two terms, the \emph{kinetic energy}
\[
\int\frac{1}{2}\abs{v}^2\rho\,dx
\]
and a \emph{potential energy}
\[
-\int \frac{\gamma^2}{2}\frac{\abs{\nabla \rho}^2}{\rho}\,dx
\]

\subsection{Equivalence of the optimal transport problems} \label{sec:equivalence}
\begin{fact}
The Schrödinger bridge problems~\eqref{eq:SB} and~\eqref{eq:SBstar} and the Yasue problem~\eqref{eq:Y} are all equivalent. In particular in terms of minimizers, the densities $\rho$ in~\eqref{eq:SB},~\eqref{eq:SBstar} and~\eqref{eq:Y} are the same.
\end{fact}
\begin{proof}
The heart of the matter is a change of variables $b\to v$. In the Schrödinger bridge problem~\eqref{eq:SB}, introduce the variable $v$ defined by
\begin{equation*}
b = v + \gamma\nabla\ln\rho
\end{equation*}
The constraint then becomes
\begin{equation*}
0 = \partial_t\rho + \div\big(\rho (b - \gamma\nabla\ln\rho\big) = \partial_t\rho + \div(\rho v)
\end{equation*}
while regarding the action functional we can compute
\begin{equation*}
\int \frac{1}{2}\abs{b}^2\rho\,dx = \int \frac{1}{2}\abs{v}^2\rho + \frac{\gamma^2}{2}\frac{\abs{\nabla \rho}^2}{\rho}\,dx + \frac{d}{dt}\calS(\rho)
\end{equation*}
where $\calS$ is the entropy defined by $\calS(\rho) = \gamma \int \rho\,\ln\rho\,dx$. Integrating in time, the action functionals in~\eqref{eq:SB} and~\eqref{eq:Y} match up to a term $\calS(\mu)-\calS(\nu)$ which only depends on the boundary terms. This proves that~\eqref{eq:Y} is nothing else but a version of~\eqref{eq:SB} under a change of variables. 

Since the Yasue problem doesn't see the direction of time (or in other words, changing $\gamma$ to $-\gamma$ doesn't change the problem) this implies that the Yasue problem~\eqref{eq:Y} is also equivalent to the backward-in-time Schrödinger bridge problem~\eqref{eq:SBstar}. 
\end{proof} 



\section{Our geometric viewpoint} \label{sec:models}
In this section we discuss our fresh perspective on the Schrödinger bridge and the Yasue problems, exploring the geometry behind these optimal transport models. We associate to the Schrödinger bridge problem a family of optimal control problems, and to the Yasue problem a family of Lagrangian mechanics problems. 

The general setting is a Riemannian manifold $(M,g)$ together with a potential function $V\colon M\to \Real$. We propose three families of variational problems on manifolds:
\begin{enumerate}[a)]
\item the optimal control problems~\eqref{eq:oc} are associated to the forward-in-time Schrödinger bridge problem~\eqref{eq:SB} ;
\item the optimal control problems~\eqref{eq:ocstar} are associated to the backward-in-time Schrödinger bridge problem~\eqref{eq:SBstar} ;
\item the Lagrangian mechanics problems~\eqref{eq:m} are associated to the Yasue problem~\eqref{eq:Y}.
\end{enumerate}
\begin{equation}
\label{eq:oc}
\tag{oc}
\begin{aligned}
& \minimize_{q,\,b} & & \int_0^1 \frac{1}{2}\abs{b_t}^2\,dt \\
& \text{subject to:} & & \dot{q}_t = b_t - \nabla V(q_t) \quad\text{for } 0<t<1\, , \\
& & & q_0 = y\,, \; q_1 = z
\end{aligned}
\end{equation}
\begin{equation}
\label{eq:ocstar}
\tag{oc*}
\begin{aligned}
& \minimize_{q\,,b^*} & & \int_0^1 \frac{1}{2}\abs{b_t^*}^2\,dt \\
& \text{subject to} & & \dot{q}_t = b_t^* + \nabla V(q_t) \quad\text{for } 0<t<1 \,, \\
& & & q_0 = y\,, \; q_1 = z
\end{aligned}
\end{equation}
\begin{equation}
\label{eq:m}
\tag{m}
\begin{aligned}
& \minimize_{q,\,v}   & & \int_0^1\frac{1}{2}\abs{v_t}^2 + \frac{1}{2}\abs{\nabla V(q_t)}^2\,dt  \\
& \text{subject to} & & \dot{q}_t = v_t \quad\text{for } 0<t<1 \,, \\
& &  & q_0 = y\,, \; q_1 = z
\end{aligned}
\end{equation}

Here $q\colon [0,1]\to M$ is a path on $M$ and for each time $t$ the control $b_t$ and the velocity $v_t$ belong to the tangent space at point $q_t$. The Riemannian norm $\abs{b}^2 = \abs{b}^2_q = g_q(b,b)$ comes from the Riemannian metric $g$ of $M$. The gradient $\nabla V$ of the scalar potential $V$ is the vector field defined for any path $(q_t)$ by: $d/dt \big(V(q_t)\big) = g_{q_t}\big(\nabla V(q_t),\dot q_t\big)$. The endpoints $y\in M$ and $z\in M$ are fixed. 

The geometric problems formally generalize the optimal transport problems: they are equal when the manifold $(M,g)$ is taken to be the Wasserstein space and the potential $V$ to be the entropy (see Correspondence Principle~\ref{prop:linkoptimaltransportgeometry}). In general however we place no restriction on the manifold or the potential which is the reason why we call these problems on manifolds \emph{families} of problems.

Note that the only difference between problems~\eqref{eq:oc} and~\eqref{eq:ocstar} is that the sign of the potential $V$ was changed. The problem~\eqref{eq:m} is a classical mechanics problem in Lagrangian formalism: one looks for critical points of an action functional
\[
\mathcal{A} = \int_0^1\frac{1}{2}\abs{v}^2 + \frac{1}{2}\abs{\nabla V(q)}^2\,dt
\]
which includes the kinetic energy $\frac{1}{2}\abs{v}^2$ and a potential energy $-\frac{1}{2}\abs{\nabla V(q)}^2$.

\begin{correspondence} \label{prop:linkoptimaltransportgeometry}
Consider the geometric problems in the special case where
\begin{enumerate}[1)]
\item the manifold $(M,g)$ is taken to be the Wasserstein space $\calP(\Rd)$ (see def.~\ref{def:wassersteinspace})
\item the potential $V$ is the entropy $$\calS(\rho) = \gamma\int\rho\ln\rho\,dx$$
\end{enumerate}
Then the optimal control problem~\eqref{eq:oc} reduces to the Schrödinger bridge problem~\eqref{eq:SB}, ~\eqref{eq:ocstar} to~\eqref{eq:SBstar} and~\eqref{eq:m} to~\eqref{eq:Y}. In particular, the correspondence is as follows

\begin{center}
\begin{tabular}{| c  c |} \hline
\textbf{Geometric problems} & \textbf{Optimal transport problems} \\ \hline
Riemannian manifold $M$ & Wasserstein space $\calP(\Rd)$ \\ \hline
point $q\in M$ & density $\rho=\rho(x)$ \\ \hline
vector $b\in T_qM$ & vector field $b(x)$ \\ \hline
velocity $v\in T_qM$ & vector field $v(x)$ \\ \hline
norm $\abs{b}_q^2$ & norm $\int \abs{b(x)}^2\,\rho(x)\, dx$ \\ \hline
gradient $\nabla V(q)$ & term $\gamma\nabla \ln\rho$ \\  \hline
endpoints $y$ and $z$ & densities $\mu$ and $\nu$ \\ \hline
\end{tabular}
\end{center}
\end{correspondence}

Recall that the three optimal transport problems~\eqref{eq:SB}, \eqref{eq:SBstar} and~\eqref{eq:Y} are all equivalent. The next proposition establishes that such an equivalence holds for the problems on manifolds, with no restriction on the manifold $(M,g)$ or the potential function $V$, and with a very elementary proof. 

\begin{prop} \label{prop:equivalencegeometricmodels}
The two geometric problems
\begin{equation}
\label{eq:oc}
\tag{oc}
\begin{aligned}
& \minimize_{q,\,b} & & \int_0^1 \frac{1}{2}\abs{b_t}^2\,dt \\
& \text{subject to} & & \dot{q}_t = b_t - \nabla V(q_t) \quad\text{for } 0<t<1 \,, \\
& & & q_0 = y\,, \; q_1 = z
\end{aligned}
\end{equation}
and
\begin{equation}
\label{eq:m}
\tag{m}
\begin{aligned}
& \minimize_{q,\,v}   & & \int_0^1\frac{1}{2}\abs{v_t}^2 + \frac{1}{2}\abs{\nabla V(q_t)}^2\,dt \\
& \text{subject to} & & \dot{q}_t = v_t \quad\text{for } 0<t<1 \,, \\
& &  & q_0 = y\,, \; q_1 = z
\end{aligned}
\end{equation}
are equivalent. The heart of the matter is a change of variables $(q,b)\to(q,v)$ with $v=b-\nabla V(q)$. 

Since replacing $V$ by $-V$ doesn't change the Lagrangian mechanics problem~\eqref{eq:m}, we can deduce that the two problems~\eqref{eq:m} and~\eqref{eq:ocstar} are also equivalent. As a consequence, all three geometric problems are equivalent.
\end{prop}
\begin{proof}
In the optimal control problem~\eqref{eq:oc} perform the change of variables $b = v + \nabla V(q)$. Then 
\begin{equation*}
\frac{1}{2} \abs{b_t}^2 = \frac{1}{2} \abs{v_t}^2 + \frac{1}{2}\abs{\nabla V(q_t)}^2 + \frac{d}{dt}V(q_t)
\end{equation*}
and integrating in time, the last term $\frac{d}{dt} V(q_t)$ only depends on the boundary terms.
\end{proof}

We would like to emphasize that Prop.~\ref{prop:equivalencegeometricmodels} says that the equivalence of the different geometric problems (and by analogy the equivalence of the optimal transport problems) is stronger than the equality of their critical points: it is essentially one minimization problem under different changes of variables $(q,b) \leftrightarrow (q,v) \leftrightarrow (q,b^*)$.

We now include a list of examples.

\begin{example}
In the geometric problems~\eqref{eq:oc} or~\eqref{eq:m}, taking the potential $V$ to be identically $0$ results in a \emph{geodesic} problem
\[
\begin{aligned}
& \minimize_{q,v} & & \int_0^1 \frac{1}{2}\abs{v_t}^2\,dt \\
& \text{subject to} & & \dot{q}_t = v_t \quad\text{for } 0<t<1 \,, \\
& & & q_0 = y\,, \; q_1 = z
\end{aligned}
\]
The optimal transport analogue is the Benamou--Brenier formula.
\end{example}

\begin{example}
In the optimal transport setting $M=\calP(\Rd)$, choose the potential
\[
\calV(\rho) = \int f(x)\rho(x)\,dx
\]
for some function $f\colon\Rd\to\Real$. The optimal control problem is
\begin{equation*}
\begin{aligned}
& \minimize_{\rho,\,b} & & \int_0^1 \!\! \int\frac{1}{2}\abs{b}^2\,\rho\,dx\,dt \\
& \text{subject to:} & & \partial_t\rho + \div\big(\rho (b-\nabla f)\big) = 0 \quad\text{for } 0<t<1 \,, \\
& & & \rho_0 = \mu\,, \; \rho_1 = \nu
\end{aligned}
\end{equation*}
while the equivalent mechanics problem reads
\begin{equation*}
\begin{aligned}
& \minimize_{\rho,\,v} & & \int_0^1 \!\! \int\frac{1}{2}\abs{v}^2\,\rho + \frac{1}{2}\abs{\nabla f}^2\,\rho\,dx\,dt \\
& \text{subject to:} & & \partial_t\rho + \div(\rho v) = 0 \quad\text{for } 0<t<1 \,, \\
& & & \rho_0 = \mu\,, \; \rho_1 = \nu
\end{aligned}
\end{equation*}
\end{example}

\begin{example}
In the optimal transport setting $M=\calP(\Rd)$, choose the potential
\[
\calV(\rho) = \gamma\int \frac{\rho^m}{m-1}\,dx
\]
with $m>1$. This is the functional considered by Otto whose gradient flow in the Wasserstein metric is the porous medium equation~\cite{ottocalculus}. The optimal control transport problem is then
\begin{equation*}
\begin{aligned}
& \minimize_{\rho,\,b} & & \int_0^1 \!\! \int\frac{1}{2}\abs{b}^2\,\rho\,dx\,dt \\
& \text{subject to:} & & \partial_t\rho + \div(\rho b) = \gamma\laplacian(\rho^m) \quad\text{for } 0<t<1 \,, \\
& & & \rho_0 = \mu\,, \; \rho_1 = \nu
\end{aligned}
\end{equation*}
There is now a nonlinear diffusion term $\gamma\laplacian(\rho^m)$. Note that taking $m\to 1$ recovers the Schrödinger bridge problem. The equivalent Lagrangian mechanics formulation is:
\begin{equation*}
\begin{aligned}
& \minimize_{\rho,\,v}   & & \int_0^1 \!\! \int\frac{1}{2}\abs{v}^2\,\rho + \frac{\gamma^2 m^2}{2}\abs{\nabla \rho}^2 \rho^{2m-3}\,dx \,dt \\
& \text{subject to:} & & \partial\rho + \div(\rho v) = 0 \quad\text{for } 0<t<1 \,, \\
& &  & \rho_0 = \mu\,, \; \rho_1 = \nu
\end{aligned}
\end{equation*}
\end{example}



\section{Hopf--Cole transformation} \label{sec:hopfcole}
In this section we propose an analogue of the Hopf--Cole transformation in the context of our variational problems on manifolds. We show in Theorem~\ref{th:mainthm} and Prop.~\ref{prop:elnox} that this transformation exists under certain restrictions on the metric and the potential, and that several remarkable properties of the optimal transport case are visible in the geometric setting. 

\subsection{Background on Hopf--Cole in the optimal transport problems} \label{sec:backgroundhopfcole}
The Euler--Lagrange equations of the Schrödinger bridge problem are given by the nonlinear system:
\begin{equation} \label{eq:eulerlagrangeschroedinger}
\begin{cases}
\partial_t\phi + \frac{1}{2} \abs{\nabla\phi}^2  =- \gamma\laplacian\phi \\
\partial_t\rho+\div(\rho\nabla\phi)=\gamma\laplacian\rho
\end{cases}
\end{equation}
where the scalar function $\phi$ is such that $b_t(x)=\nabla\phi_t(x)$, together with boundary conditions $\rho_0=\mu\,,\rho_1=\nu$. The first equation in~\eqref{eq:eulerlagrangeschroedinger} is a Hamilton--Jacobi--Bellman equation while the second one is the continuity equation. This system appears to be rather difficult since it is coupled in variables $\rho$ and $\phi$ and because of the atypical boundary conditions. Note however that the Hamilton--Jacobi--Bellman equation doesn't depend on $\rho$ ; we give a geometric analogue of this phenomenon in Prop.~\ref{prop:elnox}. Remarkably, the following Hopf--Cole change of variables
\begin{equation} \label{eq:hopfcoleschrodinger}
\begin{cases}
\eta = \exp\big(\phi/(2\gamma)\big) \\
\eta^* = \rho \exp\big(-\phi/(2\gamma)\big)
\end{cases}
\end{equation}
reduces the Euler--Lagrange equations~\eqref{eq:eulerlagrangeschroedinger} to a simple system of a backward and forward heat equation
\begin{equation} \label{eq:schrodingersystem}
\begin{cases}
\partial_t\eta = -\gamma \laplacian\eta \\
\partial_t\eta^* = \gamma \laplacian\eta^*
\end{cases}
\end{equation}

Likewise, note that since the Yasue problem is nothing but the Schrödinger bridge problem under a change of variables, we can write a version of the Hopf--Cole transformation. More specifically, the Euler--Lagrange equations of the Yasue problem take the form
\begin{equation*}
\begin{cases}
\partial_t\psi + \frac{1}{2}\abs{\nabla\psi}^2=-2\gamma^2\,(\laplacian\sqrt{\rho}) / \sqrt{\rho} \\
\partial_t\rho+\div(\rho\nabla\psi)=0
\end{cases}
\end{equation*}
where the scalar function $\psi$ is such that $v_t(x)=\nabla\psi_t(x)$. The Hopf--Cole transformation can be written as
\begin{equation} \label{eq:hopfcoleyasue}
\begin{cases}
\eta = \sqrt{\rho}\exp\big(\psi/(2\gamma)\big) \\
\eta^* = \sqrt{\rho} \exp\big(-\psi/(2\gamma)\big)
\end{cases}
\end{equation}

The heat equation system~\eqref{eq:schrodingersystem} can easily be integrated in time. Therefore locating a minimizer of the Schrödinger bridge problem amounts to finding a solution $(\eta^*_0,\eta_1)$ of the so-called Schrödinger system
\begin{equation*}
\begin{cases}
\mu(x) = \eta^*_0(x) \int K(x,y)\,\eta_1(y)\,dy \\
\nu(y) = \eta_1(y)\int K(x,y)\,\eta^*_0(x)\,dx
\end{cases}
\end{equation*}
where $K(x,y) = \frac{1}{(4\pi\gamma)^{d/2}} \exp\Big(\!-\frac{\abs{x-y}^2}{4\gamma}\Big)$ is a heat kernel. Solutions of the Schrödinger system have been studied in~\cite{fortet1940,beurling1960,jamison1975}.
 
\subsection{An analogue of the Hopf--Cole transformation for the geometric problems}
\subsubsection*{Euler--Lagrange equations of the geometric problems} \label{sec:hopfcoleanalogue}
We recall the two geometric problems~\eqref{eq:oc} and~\eqref{eq:m} introduced in Section~\ref{sec:models} and record their Euler--Lagrange equations:
\begin{fact}
The Euler--Lagrange equations of the optimal control problem
\begin{equation}
\label{eq:oc}
\tag{oc}
\begin{aligned}
& \minimize_{q,\,b} & & \int_0^1 \frac{1}{2}\abs{b_t}^2\,dt \\
& \text{subject to:} & & \dot{q}_t = b_t - \nabla V(q_t) \quad\text{for } 0<t<1\, , \\
& & & q_0 = y\,, \; q_1 = z
\end{aligned}
\end{equation}
are given by
\begin{equation} \label{eq:eulerlagrangeoc}
D_tb = \nabla_b(\nabla V)
\end{equation}
The Euler--Lagrange equations of the Lagrangian mechanics problem
\begin{equation}
\label{eq:m}
\tag{m}
\begin{aligned}
& \minimize_{q,\,v}   & & \int_0^1\frac{1}{2}\abs{v_t}^2 + \frac{1}{2}\abs{\nabla V(q_t)}^2\,dt  \\
& \text{subject to} & & \dot{q}_t = v_t \quad\text{for } 0<t<1 \,, \\
& &  & q_0 = y\,, \; q_1 = z
\end{aligned}
\end{equation}
are given by
\begin{equation} \label{eq:eulerlagrangem}
D_t v = \nabla\left(\frac{\abs{\nabla V}^2}{2}\right)
\end{equation}
\end{fact} 

Here the symbol $\nabla$ can have two related but distinct meanings:
\begin{enumerate}[a)]
\item for a \emph{scalar} function $f\colon M\to\Real$, the quantity $\nabla f$ is the \emph{gradient} of $f$. It is thus a vector field on $M$. 
\item for two \emph{vector fields} $u$ and $w$ on $M$, the quantity $\nabla_u w$ is the \emph{covariant derivative} of $w$ along $u$.
\end{enumerate}
Moreover $D_t$ denotes the covariant derivative along the path $(q_t)$ : $D_t=\nabla_{\dot{q}}\,$.

\subsubsection*{Analogue of the Hopf--Cole transformation}
The Hopf--Cole transformation in the optimal transport problems given by~\eqref{eq:hopfcoleschrodinger} or~\eqref{eq:hopfcoleyasue} maps two coupled, nonlinear PDEs into two simple uncoupled linear PDEs, a remarkable situation. In order to give a version of the Hopf--Cole transformation suited for the geometric models~\eqref{eq:oc} and~\eqref{eq:m}, we obviously need to consider restrictions on the metric $g$ and the potential $V$. The setting of this section is thus the following: assume that a global \emph{coordinate chart}
\[
(q^1,\dots,q^n)\colon M\to \Rn
\]
is introduced on the $n$-dimensional manifold $(M,g)$, define $e_i = \partial_i = \partial / \partial_{q^i}$ and $(\eps^j)_j$ the dual basis of $(e_i)_i$. 
Therefore the $e_i$'s are vector fields, the $\eps^j$'s are covector fields, and $\bracket{\eps^j,e_i} = \delta_i^j$. The restrictions that we consider on the metric $g$ and the potential $V$ take the following form:

\begin{assumption} \label{as:metric}
The quantity $\partial_i g^{jk}(q^1,\dots,q^n)$ does not depend on $(q^1,\dots,q^n)$, for any indices $i,j,k$.
\end{assumption}

\begin{assumption} \label{as:potential}
The quantity $\partial_i \big(g^{jk}(q^1,\dots,q^n)\partial_kV(q^1,\dots,q^n)\big)$ does not depend on $(q^1,\dots,q^n)$, for any indices $i,j$.
\end{assumption}
Here, as is standard notation in differential geometry, the tensor $g^{jk}$ denotes the inverse of the metric: $g_{ij} g^{jk} = \delta_i^k$. Note that the quantities present in the assumptions involve at the same time the metric, the potential and the coordinates. 

Before stating the main theorem of this section, let us highlight some features of the Euler--Lagrange equations of the optimal transport problems~\eqref{eq:SB} and~\eqref{eq:Y} (Section~\ref{sec:backgroundhopfcole}):
\begin{enumerate}[a)]
\item The Euler--Lagrange equations of the optimal transport problems are two coupled equations (in variables $(\rho,\phi)$ or $(\rho,\psi)$). The Hopf--Cole transformation results in two uncoupled equations~\eqref{eq:schrodingersystem} in variables $(\eta,\eta^*)$. 
\item Furthermore the resulting equations~\eqref{eq:schrodingersystem} are each simpler than the ones before applying the Hopf--Cole transformation.
\item The Hamilton--Jacobi--Bellman equation in~\eqref{eq:eulerlagrangeschroedinger} doesn't depend on the variable $\rho$ but only on $\phi$. 
\end{enumerate}
All these particularities have analogues in the geometric setting:

\begin{theorem} \label{th:mainthm}
Suppose that assumptions~\textup{\ref{as:metric}} and~\textup{\ref{as:potential}} hold, and that in addition
\begin{itemize}
\item the map $(q^1,\dots,q^n) \to \big(\partial_iV(q^1,\dots,q^n)\big)_i$ is invertible 
\item the matrix of partial derivatives $\big(\partial^2_{ij}V(q^1,\dots,q^n)\big)_{ij}$ is injective at each point $(q^1,\dots,q^n)\in \Rn$.
\end{itemize}
Then there exists a transformation $(q,b)\to(\eta, \eta^*)$ such that the Euler--Lagrange equations~\eqref{eq:eulerlagrangeoc} of the optimal control problem~\eqref{eq:oc} can be written as a system of a backward and a forward gradient flow on $M$:
\begin{equation*}
\begin{cases}
\dot{\eta} = \nabla V(\eta) \\
\dot{\eta^*} = -\nabla V(\eta^*)
\end{cases}
\end{equation*}
This transformation, an analogue of Hopf--Cole in the geometric problems, is given by
\begin{equation*}
\begin{cases}
\quad 2 \partial_i V(\eta) = g_{ij}(q)b^j  \\
- 2\partial_i V(\eta^*) = g_{ij}(q)b^j - 2 \partial_i V(q)
\end{cases}
\end{equation*}
Furthermore, since the problems~\eqref{eq:oc} and~\eqref{eq:m} have the same Euler--Lagrange equations (up to a change of variables), we can alternatively write for the geometric problem~\eqref{eq:m} a transformation $(q,v)\to(\eta, \eta^*)\,$ given by
\begin{equation*}
\begin{cases}
\quad 2 \partial_i V(\eta) = g_{ij}(q)v^j + \partial_iV(q) \\
-2\partial_i V(\eta^*) = g_{ij}(q)v^j - \partial_iV(q)
\end{cases}
\end{equation*}

\end{theorem} 

We delay the proof of this theorem until Section~\ref{sec:proofs}. This result gives an analogue of points a) and b) in the geometric setting. Indeed note that after the transformation the equation in $\eta$ (resp. $\eta^*$) is simpler since it is a first order ODE with a gradient flow structure. Moreover the assumptions~\ref{as:metric} and~\ref{as:potential} are also sufficient to generalize point c):

\begin{prop} \label{prop:elnox}
If Assumptions~\textup{\ref{as:metric}} and~\textup{\ref{as:potential}} hold, then the Euler--Lagrange equations~\eqref{eq:eulerlagrangeoc} of the optimal control problem~\eqref{eq:oc} do not depend on $q$ when written in coordinates. 

More precisely, define $\phi=\flat(b) = g(b,\cdot)$ the covector field obtained by lowering indices of $b$. The Euler--Lagrange equations~\eqref{eq:eulerlagrangeoc} written in terms of $\phi$ are
\begin{equation} \label{eq:eulerlagrangeocphi}
\dot\phi_i + \tfrac{1}{2} \, (\partial_ig^{jk})\phi_j\phi_k = \partial_i(g^{jk}\partial_jV)\phi_k
\end{equation}
Note that the two terms $\tfrac{1}{2} \, (\partial_ig^{jk})$ and $\partial_i(g^{jk}\partial_jV)$ are precisely the ones covered by Assumptions~\textup{\ref{as:metric}} and~\textup{\ref{as:potential}}: they are thus constants that do not depend on $q$.
\end{prop}

We delay the proof of this proposition until Section~\ref{sec:proofs}. Note that the assumptions~\textup{\ref{as:metric}} and~\textup{\ref{as:potential}} are not exactly geometric in nature because they involve a coordinate system. This seems to indicate that the expression of the entropy in \emph{Eulerian coordinates} plays a role in the existence of the Hopf--Cole transformation. However, while the transformation in Theorem~\eqref{th:mainthm} relies on coordinates, the resulting equations on $\eta$ and $\eta^*$ are themselves geometric (a backward and a forward gradient flow).

\subsubsection*{Background on calculus in Wasserstein space}
As preparation for stating the next section's correspondence principle, let us review some concepts introduced by Lott in~\cite{lott08} about calculus in Wasserstein space. First a definition:
\begin{definition} \label{def:wassersteinspace}
The Wasserstein space $\calP(\Rd)$ is the space of Borel probability measures on $\Rd$, equipped with the Wasserstein metric $W_2$. 

Sometimes it is needed to work instead with the subspace
\[
\calP^{\infty}(\Rd) = \set{\rho \, dx \,\middle|\, \rho\in C^{\infty}(\Rd),\, \rho > 0,\, \int\rho\,dx  =1}
\]
\end{definition}

Given $f\in C^{\infty}(\Rd)$ Lott defines $F_f\in C^{\infty}(\calP(\Rd))$ by
\[
F_f(\rho) = \int f(x)\rho(x)\,dx
\]
The role of the functions $F_f$ is close to the one of \emph{Eulerian coordinates} (think of $f = \delta_x$ for some point $x\in\Rd$). Additionally Lott defines a vector field $V_f$ on the Wasserstein space $\calP(\Rd)$ such that for all $F\in C^{\infty}\big(\calP(\Rd)\big)$,
\[
\big(V_f F\big)(\rho) = \left.\frac{d}{d\eps} \right|_{\eps=0} F\Big(\rho - \eps \div(\rho\nabla f)\Big)
\]
The Riemannian metric on the Wasserstein space $\calP(\Rd)$ is then given by

\begin{equation} \label{eq:wassersteinmetriclott}
\bracket{V_f,V_g}_{\rho} = \int \nabla f \cdot \nabla g\, \rho\,dx = -\int\div(\rho\nabla f)\,g\,dx
\end{equation}

\subsubsection*{Correspondence result}
We would like now to bridge the gap between Theorem~\ref{th:mainthm} (which deals with the geometric problems) and the optimal transport setting. Without loss of generality, we will focus on the optimal control problems (i.e.~\eqref{eq:oc} in the geometric setting and~\eqref{eq:SB} in the optimal transport setting) --- recall that  their Euler--Lagrange equations are equivalent to those of the Lagrangian mechanics problems (i.e.~\eqref{eq:m} and~\eqref{eq:Y}). First, as in Prop.~\ref{prop:elnox}, it is a good idea to consider instead of the vector $b\in TM$ the covector $\phi=\flat(b) = g(b,\cdot)$ obtained by lowering indices. The correspondence principle below explains precisely why the Euler--Lagrange equation~\eqref{eq:eulerlagrangeocphi} of~\eqref{eq:oc} 
\[
\dot\phi_i + \tfrac{1}{2} \, (\partial_ig^{jk})\phi_j\phi_k = \partial_i(g^{jk}\partial_jV)\phi_k
\]
corresponds to the Hamilton--Jacobi--Bellman equation in~\eqref{eq:eulerlagrangeschroedinger}
\[
\partial_t\phi + \frac{1}{2}\abs{\nabla\phi}^2 = -\gamma\laplacian\phi
\]
or why the transformation
\begin{equation*}
\begin{cases}
\quad 2 \partial_i V(\eta) = \phi_j  \\
- 2\partial_i V(\eta^*) = \phi_j - 2 \partial_i V(x)
\end{cases}
\end{equation*}
does indeed correspond to the Hopf--Cole transformation in the Schrödinger bridge problem, which can be written as 
\begin{equation*}
\begin{cases}
\quad 2\gamma\ln\eta = \phi \\
-2\gamma\ln\eta^* = \phi - 2\gamma\ln\rho

\end{cases}
\end{equation*}

\begin{correspondence} \label{prop:linkoptimaltransportgeometry2}
Consider the geometric setting in the special case where 
\begin{enumerate}[1)]
\item the manifold $(M,g)$ is taken to be the subspace $\calP^{\infty}(\Rd)$ (see def.~\ref{def:wassersteinspace})
\item the potential $V$ is the entropy $\calS(\rho) = \gamma\int\rho\ln\rho\,dx$
\end{enumerate}
Then the correspondence between the geometric and optimal transport settings is as follows:
\begin{center}
\begin{tabular}{| c  c |} \hline
\textbf{Geometric problems} & \textbf{Optimal transport problems} \\ \hline
point $q\in M$ & probability density $\rho$ \\ \hline
index $i$, $1\le i\le n$ & Eulerian point $x\in\Rd$ \\ \hline
coordinate $q^i$ & Eulerian coordinate $\rho(x)$ \\ \hline
$\partial_i V(q)$ & $\gamma\ln\rho(x)$ \\ \hline
$d\phi / dt$ & $\partial_t\phi$ \\ \hline
$(\partial_ig^{jk})\phi_j\phi_k$ & $\abs{\nabla\phi}^2(x)$ \\ \hline
$\partial_i(g^{jk}\partial_jV)\phi_k$ & $-\gamma\laplacian\phi(x)$ \\ \hline
\end{tabular}
\end{center}

In particular, the assumptions~\ref{as:metric} and~\ref{as:potential} as well as the two assumptions in Theorem~\ref{th:mainthm} are verified.
\end{correspondence}

\begin{proof}
In Lott's notation, the expression of the Wasserstein space metric is given by~\eqref{eq:wassersteinmetriclott}. 
In particular, if $(\rho_t)$ is a time-dependent density, the square norm of its velocity is 
\begin{equation} \label{eq:eqproof1}
\int\partial_t\rho\,\phi\,dx
\end{equation}
where $\phi$ is defined (up to an additive constant) by $-\div(\rho\nabla\phi) = \partial_t\rho$. On a general Riemannian manifold $M$, if $(q_t)$ is a path then the square norm $\abs{\dot{q}}^2_{q}$ of its velocity is given in coordinates by 
\begin{equation} \label{eq:eqproof2}
\dot{q}^i \, g_{ij}\dot{q}^j
\end{equation}
Comparing the two expressions, we deduce that the term $\phi$ in~\eqref{eq:eqproof1} corresponds to the term $g_{ij}\dot{q}^j$ in~\eqref{eq:eqproof2}. In other words, given a path $(\rho_t)$ in the Wasserstein space equipped with Eulerian coordinates, the lowering of indices of the velocity $\partial_t\rho(x)$ is given by $\phi(x)$. Thus $\phi$ can be seen as a covector in Eulerian coordinates. 

We now prove the penultimate line of the table. The previous discussion implies that the square norm $\abs{\dot{q}}^2_{q}$ of the velocity can be written in terms of $\phi_i:=g_{ij}(q)\,\dot{q}^j$ on a manifold $M$ as
\[
g^{jk}(q)\phi_j\phi_k
\]
or on the Wasserstein space as
\[
\int\abs{\nabla\phi}^2\,\rho\,dx
\]
Keeping $\phi$ constant in coordinates (resp. Eulerian coordinates) and taking derivatives with respect to $q$ (resp. $\rho$) we can conclude that the term $\partial_ig^{jk} \phi_j\phi_k$ corresponds to $\abs{\nabla\phi}^2(x)$. 

The last line of the table can be proven in a similar fashion. When the potential $V=V(q)$ is taken to be the entropy $\calS(\rho)=\gamma\int\rho\ln\rho$ on the Wasserstein space, taking derivatives with respect to $q$ (resp. $\rho$) implies that the term $\partial_jV(q)$ corresponds to $\gamma\ln\rho(x)$. By raising indices, the term $g^{jk}\,\partial_jV(q)\,\phi_k$ corresponds to $\int \nabla(\gamma\ln\rho)\cdot \nabla \phi \,\rho\,dx = \int\gamma\nabla\rho\cdot\nabla\phi\,dx = -\int \gamma\laplacian \phi \,\rho\,dx$. Finally, taking derivatives with respect to $q$ (resp. $\rho$) and keeping $\phi$ constant in coordinates (resp. Eulerian coordinates) implies that 
\[
\partial_i\big(g^{jk}\,\partial_jV(q)\big)\,\phi_k
\]
corresponds to 
\[
-\gamma\laplacian\phi(x)
\]
Therefore assumptions~\ref{as:metric} and~\ref{as:potential} and the assumptions in Theorem~\ref{th:mainthm} are verified, at least on the subspace $\calP^{\infty}(\Rd)$ (defined in def.~\ref{def:wassersteinspace}).

Indeed, the quantities $\abs{\nabla\phi}^2(x)$ and $-\gamma\laplacian\phi(x)$ do not depend on $\rho$, thus Assumption~\ref{as:metric} and~\ref{as:potential} are verified. Additionally, the transformation $\phi\to\eta$ given by $2\gamma\ln\eta = \phi$ does define $\eta$ uniquely, thus the first assumption of Theorem~\ref{th:mainthm} is verified. Finally, taking two derivatives of the potential in coordinates (resp. the entropy in Eulerian coordinates), we see that the Hessian $(\partial_{ij}^2V)_{ij}$ in Theorem~\ref{th:mainthm} corresponds to pointwise multiplication by  $\frac{1}{\rho(x)}$, or more precisely corresponds to the infinite-dimensional matrix $\left(\frac{1}{\rho(x)} \delta_{x=y}\right)_{xy}$ which is injective. Thus the second assumption of Theorem~\ref{th:mainthm} is also verified.
\end{proof}

\subsection{Proofs} \label{sec:proofs}
In this section, we prove Prop.~\ref{prop:elnox} and Theorem~\ref{th:mainthm}.

\begin{proof}[Proof of Prop.~\ref{prop:elnox}]
We begin by defining $\phi = \flat(b) = g(b,\cdot)$ where $\flat$ is the flat operator from differential geometry which lowers indices: in coordinates $\phi_i = g_{ij}b^j$. Recall now that the Euler--Lagrange equations~\eqref{eq:eulerlagrangeoc} are given on the manifold $M$ by $D_tb = \nabla_b(\nabla V)$. Since the covariant derivative and the musical isomorphisms $\flat$ and $\sharp$ commute, the switch to $\phi$ variables is given by
\[
D_t\phi = \nabla_{\sharp\phi}(DV)
\]
where $DV$ is the differential of $V$ (a covector), and $\sharp\colon T^*M\to TM$ denotes the sharp operator from differential geometry which raises indices. In coordinates, this last equation can be written
\[
\dot{\phi_i} - \Gamma^k_{ij}\,\dot{q}^j\phi_k = g^{j\ell}\phi_\ell(\partial_{ij}V - \Gamma^k_{ij}\, \partial_kV)
\]
where the $\Gamma^k_{ij}$'s denote the Christoffel symbols of the second kind. The velocity $\dot{q}$ can be written in terms of $\phi\,$: $\dot{q}^j = g^{j\ell}(\phi_\ell-\partial_{\ell}V)$. This gives the equation
\[
\dot{\phi_i} - \Gamma^k_{ij}\,g^{j\ell}\phi_k\phi_\ell  = g^{j\ell}\phi_\ell\partial_{ij}V - \Gamma^k_{ij}\,g^{j\ell} (\phi_\ell \partial_kV + \phi_k\partial_{\ell}V)
\]
Note that both terms involving the quantity $\Gamma^k_{ij}\,g^{j\ell}$ are symmetric in $k$ and $\ell$, thus we can write the symmetrized version
\[
\dot{\phi_i} - \frac{1}{2}(\Gamma^k_{ij}\,g^{j\ell} + \Gamma^\ell_{ij}\,g^{jk})\, \phi_k\,\phi_\ell  = g^{j\ell}\phi_\ell\,\partial_{ij}V -\frac{1}{2}(\Gamma^k_{ij}\,g^{j\ell} + \Gamma^\ell_{ij}\,g^{jk}) (\phi_\ell \partial_kV + \phi_k\partial_{\ell}V)
\]
and we can then break the symmetry in the RHS:
\[
\dot{\phi_i} - \frac{1}{2}(\Gamma^k_{ij}\,g^{j\ell} + \Gamma^\ell_{ij}\,g^{jk})\, \phi_k\,\phi_\ell  = g^{j\ell}\phi_\ell\partial_{ij}V -(\Gamma^k_{ij}\,g^{j\ell} + \Gamma^\ell_{ij}\,g^{jk}) \phi_\ell \partial_kV 
\]
Finally, we can simplify $\Gamma^k_{ij}\,g^{j\ell} + \Gamma^\ell_{ij}\,g^{jk} = -\partial_ig^{k\ell}$, so that we may write
\[
\dot{\phi_i} + \frac{1}{2}\partial_ig^{k\ell} \phi_k\phi_\ell  = g^{j\ell}\phi_\ell\partial_{ij}V +\partial_ig^{k\ell} \phi_\ell \partial_kV 
\]
Rearranging indices and factoring the RHS leads to the desired expression
\[
\dot\phi_i + \frac{1}{2} \, \partial_ig^{jk}\phi_j\phi_k = \partial_i(g^{jk}\partial_jV)\phi_k
\]
\end{proof}

\begin{proof}[Proof of Theorem~\ref{th:mainthm}]
In this proof we use the covector $\phi = \flat(b)$ instead of the vector $b$ like in the proof of Prop.~\ref{prop:elnox}. Because of the assumptions made in the theorem, the transformation 
\begin{equation*}
2 \, \partial_i V(\eta) = \phi_i 
\end{equation*}
defines a bijection $\colon\Rn\to\Rn, (\phi_i)\to(\eta^i)$. Taking time-derivatives of this equation and using the expression of $\dot{\phi}$ from Prop.~\ref{prop:elnox}, we get
\[
2\,\partial^2_{ij}V(\eta)\,\dot{\eta}^j = -\frac{1}{2} \, \partial_ig^{jk}\phi_j\phi_k + \partial_i(g^{jk}\partial_jV)\phi_k
\]
and replacing $\phi$ by its expression in terms in $\eta$ yields
\[
2\,\partial^2_{ij}V(\eta)\,\dot{\eta}^j = -2 \, \partial_ig^{jk}\partial_jV(\eta)\,\partial_kV(\eta) + 2\partial_i(g^{jk}\partial_jV)\partial_kV(\eta)
\]
Now we make use of the assumptions~\ref{as:metric} and~\ref{as:potential} to write $\partial_ig^{jk} = \partial_ig^{jk}(\eta)$ and $\partial_i(g^{jk}\partial_jV) = \partial_i(g^{jk}\partial_jV)(\eta) = \partial_i g^{jk}\,\partial_jV(\eta) + g^{jk}(\eta)\,\partial^2_{ij}V(\eta)$. Two terms in the RHS cancel and is only left
\[
\partial^2_{ij}V(\eta)\,\dot{\eta}^j =  \partial^2_{ij}V(\eta) \, g^{jk}(\eta)\partial_kV(\eta)
\]
Since we assumed that the matrix $\big(\partial^2_{ij}V(\eta)\big)_{ij}$ was injective, the desired result on $\eta$ follows.

The proof for $\eta^*$ is nearly identical, up to a minus sign, once we remark that $D_t\phi^* = -\nabla_{\sharp\phi^*}(DV)$ where $\phi^* = \phi - 2DV(q)$.

\end{proof}

\section*{Acknowledgments}
The author would like to thank Alfred Galichon for introducing him to the subject and constant interest in this work, Robert V. Kohn for very helpful comments and Montacer Essid for fruitful discussions.

\bibliographystyle{amsalphaurl}
\bibliography{../biblio/bib/reg_ot}

\end{document}